\documentclass{article}
\usepackage{amsmath,amssymb,amsthm,verbatim,algorithm2e,url}

\begin{document}

\theoremstyle{plain}
\newtheorem{thm}{Theorem}[section]
\newtheorem{lemma}[thm]{Lemma}
\newtheorem{prop}[thm]{Proposition}
\newtheorem{cor}[thm]{Corollary}

\title{Cryptanalysis of three matrix-based key establishment protocols}
\author{Simon R. Blackburn, Carlos Cid and Ciaran Mullan\thanks{This work was supported by E.P.S.R.C. PhD studentship EP/P504309/1.}\\
Information Security Group,\\Royal Holloway, University of London, \\ Egham, Surrey, TW20 0EX, United Kingdom. \\\texttt{\{s.blackburn,carlos.cid,c.mullan\}@rhul.ac.uk}}
\date{\today}

\newcommand{\A}{\mathcal{A}}
\newcommand{\B}{\mathcal{B}}
\newcommand{\G}{\mathcal{G}}
\renewcommand{\L}{\mathcal{L}}
\newcommand{\U}{\mathcal{U}}
\newcommand{\bF}{\mathbb{F}}
\newcommand{\bZ}{\mathbb{Z}}
\newcommand{\SL}{\mathrm{SL}}
\newcommand{\GL}{\mathrm{GL}}
\newcommand{\Mat}{\mathrm{Mat}}
\newcommand{\bM}{\begin{smallmatrix}}
\newcommand{\eM}{\end{smallmatrix}}

\maketitle
\begin{abstract}
We cryptanalyse a matrix-based key transport protocol due to Baumslag, Camps, Fine, Rosenberger and Xu from 2006.
We also cryptanalyse two recently proposed matrix-based key agreement protocols, due to Habeeb, Kahrobaei and Shpilrain, and due to Romanczuk and Ustimenko.
\end{abstract}

\section{Introduction}
Regular proposals are made to employ groups in cryptography; see for example the survey article by Blackburn et al~\cite{Blackburn} or the book by Myasnikov et al~\cite{Myasnikov}.
In particular, matrix groups are often considered because matrices are easy to represent and manipulate.
However, such proposals generally have a poor reputation: we are unaware of any fully specified proposals that are widely regarded as secure.

In this paper we cryptanalyse a matrix-based key transport protocol due
to Baumslag, Camps, Fine, Rosenberger and Xu~\cite{Baumslag}, which we refer to as the BCFRX scheme.
In fact, their proposal is more general and they suggest several platform groups; we consider their only matrix group proposal.
We cryptanalyse this scheme in a very strong sense.
We show that for practical parameter sizes a passive adversary can feasibly recover the session key after observing just one run of the protocol.
We find an even more efficient attack if two or more runs of the protocol are observed.
Our techniques reduce the problem of breaking the scheme to a sequence of feasible Gr\"{o}bner basis computations.
This work constitutes Section 2.

We also cryptanalyse two recently proposed matrix-based key agreement protocols, due to Habeeb, Kahrobaei and Shpilrain (HKS)~\cite{HKS}, and due to Romanczuk and Ustimenko (RU)~\cite{Romanczuk}.
These schemes both fail due to straightforward linearisation attacks.
This work constitutes Sections 3 and 4.

\section{The BCFRX Scheme}
We begin by describing the BCFRX scheme.
The protocol assumes that Alice and Bob a priori share some secret information, namely their long-term secret key.
The goal of the protocol is for Alice and Bob to establish a session key for subsequent cryptographic use.
To achieve this, Bob chooses the session key and sends it to Alice in three passes, as follows.

Let $\G$ be a finitely presented group.
Let $\A$ and $\B$ be two commuting subgroups of $\G$ (so $AB=BA$ for all $A\in\A$ and $B\in\B$).
The group $\G$ is made public and the subgroups $\A$ and $\B$ form Alice and Bob's long-term secret key.
Then:
\begin{itemize}
\item Bob chooses a session key $K \in \G$ and elements $B,B'\in \B$.
He sends $C:=BKB'$ to Alice.
\item Alice picks elements $A,A' \in \A$ and sends $D:=ACA' = ABKB'A'$
to Bob.
\item Since $\A$ and $\B$ commute, we have that $ABKB'A' = BAKA'B'$.
Bob sends $E:=B^{-1}D{B'}^{-1}=AKA'$ to Alice.
\item Alice computes $K = A^{-1}E{A'}^{-1}$.
\end{itemize}

We can think of this protocol as Shamir's three-pass (or no-key) protocol~\cite[Protocol~12.22, Page~500]{MOV}, with the operation of multiplying on the left and right by a group element replacing the exponentiation operation.

There was no detailed discussion of security in~\cite{Baumslag}, but we need to specify a security model and what it means to break the protocol, in order to cryptanalyse it.
We will consider the weakest possible notion of security: the passive adversary model.
So we will regard the protocol as broken if we can construct an adversary that can feasibly compute the session key, after eavesdropping on one or more runs of the protocol; this adversary must perform well for practical parameter sizes.

Baumslag et al.~\cite{Baumslag} suggested several abstract platform groups to serve for~$\G$.
But in this paper we concentrate on their only matrix group proposal: $\G = \SL_4(\bZ)$, the group of invertible $4\times 4$ matrices of determinant 1 over the integers.
It was proposed that the commuting subgroups $\A$ and $\B$ should be constructed as follows.
Writing $I_2$ for the $2\times 2$ identity matrix, define the subgroups $\U$ and $\L$ of $\G$ by
\begin{equation}
\label{eqn:UL_definition}
\U=\begin{pmatrix}
\SL_2(\bZ)&0\\
0&I_2
\end{pmatrix}
\text{ and }
\L=\begin{pmatrix}
I_2&0\\
0&\SL_2(\bZ)
\end{pmatrix}.
\end{equation}
Let $M\in \SL_4(\bZ)$ be a secret matrix known to both Alice and Bob. Then we define
\begin{equation}
\label{eqn:AB_definition}
\A = M^{-1} \U M \text{ and } \B = M^{-1}\L M.
\end{equation}
We may thus view the long-term secret key as the matrix $M$.

As described the proposal is not yet fully specified, since it remains to specify how the long-term secret key $M$ is chosen, and how the protocol chooses elements from $\A$ and $\B$ at various points.
It was stated in Baumslag et al.~\cite{Baumslag} that elements are picked randomly from $\A$ and $\B$, and we presume that the matrix $M$ is picked in a similar fashion from $\G = \SL_4(\bZ)$.
But since the group $\G$ and its subgroups $\A,\B$ are infinite, the meaning of the word random is unclear in this context.
Any practical cryptanalysis will depend on the details of how these random choices are made; however the cryptanalysis we give below will work for any efficient method for making these random choices that we can think of.

In any fully specified implementation of the protocol, there exists an integer $\Lambda$ such that the entries of all matrices generated in the protocol lie in the interval $(-\Lambda/2,\Lambda/2)$.
Since the standard way to represent a $4\times 4$ integer matrix of this form uses approximately $16\log_2 \Lambda$ bits, it is natural to think of $\log_2 \Lambda$ as the security parameter of the scheme.

\subsection*{A cryptanalysis}

Our cryptanalysis proceeds in three stages.
In Stage~1, we argue that integer computations may be replaced by computations modulo $p$ for various small primes $p$.
In Stage~2 we show that knowledge of a matrix $N$ of a restricted form allows a passive adversary to compute any session key transmitted under the scheme.
Finally, in Stage~3, we show that this matrix $N$ may be computed in practice.
None of these stages is rigorous (though Stage~2 may be made so), but the stages all work well in practice.

\subsubsection*{Stage~1: Working modulo $p$}

Suppose an adversary wishes to discover a session key $K$.
Since the entries of $K$ lie in the interval between $-\Lambda/2$ and $\Lambda/2$, it is enough to find $K \bmod{n}$ for any $n>\Lambda$.
Indeed, this is how we approach our cryptanalysis.
We will show (see Stages~2 and~3 below) that in practice we may efficiently compute $K\bmod{p_i}$ for small primes $p_i$ of our choice.
(We are thinking of $p_i$ as a prime of between 80 and 300 bits in length: in some sense quite large, but in general smaller than $\Lambda$.)
We run this computation for several different primes $p_i$ until $\prod{p_i}>\Lambda$.
Setting $n = \prod{p_i}$, we can then appeal to the Chinese remainder theorem to calculate $K\bmod{n} = K$.

We write this more precisely as follows.
Let $T$ be a fully specified version of the BCFRX protocol, with $\SL_4(\bZ)$ as a platform.
For a prime $p$, let $\bZ_p$ be the integers modulo $p$.
Let $T_p$ be the BCFRX protocol under the platform group $\G=\SL_4(\bZ_p)$, defined as follows.
We identify the subgroups $\U$ and $\L$ defined by~\eqref{eqn:UL_definition} with their images in $\SL_4(\bZ_p)$.
Let the subgroups $\A$ and $\B$ be chosen to be of the form~\eqref{eqn:AB_definition} for some matrix $M\in\G$ chosen uniformly at random.
Let the protocol pick all elements from $\A$ and $\B$ uniformly and independently at random.
This makes sense since $\G$ is finite. We use $T_p$ to model the protocol $T$ taken modulo $p$.
This model is not quite accurate: for example, it is almost certain that when $M\in\SL_4(\bZ)$ is chosen according to the method specified in $T$, the distribution of $M\bmod p$ will not be quite uniform in $\SL_4(\bZ_p)$.
But for all ways we can think of in which $T$ can be specified, the protocol $T_p$ is a good model for $T$ taken modulo $p$ (in the sense that an adversary that succeeds in practice to recover the session key generated by $T_p$ will also succeed in practice to recover $K \bmod p$ when presented with the matrices from a run of the protocol $T$). Note that an adversary has great freedom in choosing $p$, which makes the reduction to $T_p$ difficult to design against.
The fact (see below) that the session key for $T_p$ can be feasibly computed in practice shows that $T$ is insecure.

\subsubsection*{Stage~2: Restricting the long-term key}

We consider the protocol $T_p$ over $\SL_4(\bZ_p)$ defined above.
From now on, let us write an arbitrary $4\times 4$ matrix $Z$ in block form as $Z = \left( \bM Z_{11} & Z_{12} \\ Z_{21} & Z_{22} \eM \right)$, for the obvious $2\times 2$ submatrices $Z_{ij}$ of $Z$.

The following lemma shows that there are many equivalent long-term keys for the protocol $T_p$.

\begin{lemma}
\label{lem:equiv_keys}
Let $M\in \SL_4(\bZ_p)$ be the long-term key shared by Alice and Bob, and define subgroups $\A$ and $\B$ by $\A= M^{-1}\U M$ and $\B= M^{-1}\L M$.
Let $N\in \GL_4(\bZ_p)$ be any matrix such that $N^{-1}\U N =\A$ and $N^{-1}\L N =\B$.
If $N$ is known, then any session key can be efficiently computed by a passive adversary.
\end{lemma}
\begin{proof}
An adversary is presented with matrices $C$, $D$ and $E$ that are transmitted as part of the protocol.
We have that $C=BKB'$, $D=ABKB'A'$ and $E=AKA'$ for some unknown matrices $A,A'\in \A$ and $B,B'\in \B$.
Suppose that the adversary is also able to obtain a matrix $N$ satisfying the conditions of the lemma.
Since $A,A'\in\A$ we may write $A=N^{-1}RN$ and $A'=N^{-1}R'N$ for some unknown matrices $R,R'\in\U$.
Similarly we may write $B=N^{-1}SN$ and $B'=N^{-1}S'N$ for some unknown matrices $S,S'\in\L$.

Define an (unknown) matrix $K'$ by $K'=NKN^{-1}$.
Define matrices $C'$, $D'$, $E'$ by
\begin{align*}
C'&:=NCN^{-1}=SK'S',\\
D'&:=NDN^{-1}=RSK'S'R'\text{ and }\\
E'&:=NEN^{-1}=RKR'.
\end{align*}
Note that the adversary can compute $C'$, $D'$ and $E'$.

Using the fact that $S,S'\in\L$ and $R,R'\in\U$, we may write
\begin{align*}
C'&=\begin{pmatrix}
K'_{11}&K'_{12}S'_{22}\\
S_{22}K'_{21}&S_{22}K'_{22}S'_{22}
\end{pmatrix},\\
D'&=\begin{pmatrix}
R_{11}K'_{11}R'_{11}&R_{11}K'_{12}S'_{22}\\
S_{22}K'_{21}R'_{11}&S_{22}K'_{22}S'_{22}
\end{pmatrix}\text{ and }\\
E'&=\begin{pmatrix}
R_{11}K'_{11}R'_{11}&R_{11}K'_{12}\\
K'_{21}R'_{11}&K'_{22}
\end{pmatrix}.
\end{align*}
Clearly $K'_{11}$ is known to the adversary, since $K'_{11}=C'_{11}$.
Moreover, $K'_{22}$ is known since $K'_{22}=E'_{22}$.

To compute $K'_{12}$, find any matrix $X$ such that $XD'_{12} = C'_{12}$ (note there may be more than one such $X$ if $K'_{12}$ is noninvertible).
This implies $XR_{11}K'_{12}=K'_{12}$, since $S'_{22}$ is invertible.
Thus an adversary can compute $XE'_{12}=K'_{12}$.
Similarly, to compute $K'_{21}$ find any matrix $Y$ such that $D'_{21}Y = C'_{21}$.
This implies $K'_{21}R'_{11}Y=K'_{21}$ and an adversary can compute $E'_{21}Y=K'_{21}$.

Once $K'$ is known, the session key $K$ may be recovered since $K=N^{-1}K'N$.
\end{proof}

Let $\Mat_{2}(\bZ_p)$ be the set of $2\times 2$ matrices over $\bZ_p$.
Let $\mathcal{I}\subseteq\Mat_2(\bZ_p)$ be defined by
\[
\mathcal{I}=\left\{ \begin{pmatrix} 1&0\\
0&1\end{pmatrix},\begin{pmatrix}1&0\\0&0\end{pmatrix},
\begin{pmatrix}0&0\\0&0\end{pmatrix}\right\}.
\]
We say that $N\in \GL_4(\bZ_p)$ is of \emph{restricted form} if $N_{11},N_{22}\in \mathcal{I}$.

\begin{lemma}
\label{lem:restricted_form}
For any long-term key $M$ used in the protocol $T_p$, there is a matrix $N$ of restricted form satisfying the conditions of Lemma~\ref{lem:equiv_keys}.
Moreover, for an overwhelming proportion of long-term keys $M$, we may impose the condition that $N_{11}=N_{22}=I_2$, where $I_2$ is the $2\times 2$ identity matrix.
\end{lemma}
\begin{proof}
Let $f:\Mat_2(\bZ_p)\rightarrow \GL_2(\bZ_p)$ be a function such that
$f(X)X\in \mathcal{I}$ for all $X\in \Mat_2(\bZ_p)$. Such a function $f$ certainly exists: it can be derived from a standard row reduction algorithm.

Define
\[
H:=\begin{pmatrix}
f(M_{11})&0\\
0&f(M_{22})
\end{pmatrix}\text{ and }N:=HM.
\]

The definition of $H$ means that $N_{11},N_{22}\in \mathcal{I}$, and so $N$ is of restricted form.
Also, any matrix
\[
H\in\begin{pmatrix}
\GL_2(\bZ_p)&0\\
0&\GL_2(\bZ_p)
\end{pmatrix}
\]
has the property that $H^{-1}\U H=\U$ and $H^{-1}\L H=\L$. So
\[
N^{-1}\U N =M^{-1}H^{-1}\U HM = M^{-1} \U M=\A
\]
and similarly $\B=N^{-1}\L N$.
So the main statement of the lemma is proved.
To see why the last statement of the lemma holds, note that for an overwhelming proportion of long-term keys $M$ we have that $M_{11}$ and $M_{22}$ are invertible. The function $f$ maps any invertible matrix to its inverse, and so $N_{11}=N_{22}=I_2$ in this case.
\end{proof}

\subsubsection*{Stage~3: Computing the matrix $N$}

We may compute an equivalent long-term key $N$ of restricted form as follows.
After eavesdropping on a run of the protocol, we know the matrices $C,D,$ and $E$.
We also know a matrix $N$ of restricted form must satisfy the equations
\begin{align}
 NDN^{-1} &= RNCN^{-1}R', \label{eqn:C,D}\\
 NDN^{-1} &= SNEN^{-1}S', \label{eqn:E,D}\\
 NN^{-1} &= I_4, \label{eqn:inv}
\end{align}
for unknown matrices $R,R'\in \U$ and $S,S'\in \L$.
Since $N$ is of restricted form we have $N_{11},N_{22}\in \mathcal{I}$.
There are thus only 9 possible combinations for $N_{11}$ and $N_{22}$, so we may perform a trivial exhaustive search to find the right combination. (In practice we would first try $N_{11}=N_{22}=I_2$, since this holds with overwhelming probability.)
We assign variables $x_1,\ldots ,x_8$ for the remaining unknown entries of $N$, and $x_9,\ldots ,x_{24}$ for the unknown entries of $N^{-1}$.

Expanding~(\ref{eqn:C,D}) and~(\ref{eqn:E,D}), we find $${(NDN^{-1})}_{22}={(NCN^{-1})}_{22}, \;\; {(NDN^{-1})}_{11}={(NEN^{-1})}_{11}.$$
This gives us $4+4=8$ quadratic equations in the $x_i, i=1,\ldots ,24$.
Adding the 16 quadratic equations from~(\ref{eqn:inv}), we have a system of 24 quadratic equations in 24 unknowns and expect a Gr\"{o}bner basis calculation to reveal $N$.
If we eavesdrop on a second run of the protocol, we learn 8 new equations (from~(\ref{eqn:C,D}) and~(\ref{eqn:E,D})) and expect to compute $N$ even more efficiently.

\subsubsection*{Experimental results}

Over 1,000 trials using Magma~\cite{Magma} Version~2.16-11 on a Intel Core 2 Duo 1.86GHz desktop, it took roughly 12 seconds to compute each (lex ordered) Gr\"{o}bner basis for a random 300-bit prime.
In all our experiments, twenty three of the basis elements had the form $x_i + f_i(x_{24})$ for $i = 1,\ldots ,23,$ where $f_i$ is a polynomial of degree 5.
The final basis element was a degree 6 polynomial in $x_{24}$.
Thus in all our cases we had a maximum of six possibilities for a matrix $N$ of restricted form satisfying Lemma~\ref{lem:equiv_keys}.

If we eavesdrop on a second run of the protocol, we can add 8 new equations (or just one of the 8 new equations) to our system.
A Gr\"{o}bner basis calculation then reveals a unique value for $N$.

\section{The HKS Scheme}

Next we turn our attention to a key agreement protocol proposed by
Habeeb, Kahrobaei and Shpilrain~\cite{HKS}.
Our description of the scheme is somewhat simplified from~\cite{HKS}.

Let $A$ be a group and let $B$ be an abelian group.
Let $\mathrm{Aut}(B)$ denote the automorphism group of $B$,
and let $A,B,\mathrm{Aut}(B),a\in A, b\in B, n\in \mathbb{N}$ be public. Then
\begin{itemize}
 \item Alice picks an embedding $\psi :A\rightarrow \mathrm{Aut}(B)$ and sends\\ $x = \psi(a)(b)\psi(a^2)(b)\ldots \psi(a^{n-1})(b)$ to Bob.
 \item Bob picks an embedding $\phi :A\rightarrow \mathrm{Aut}(B)$ and sends\\ $y = \phi(a)(b)\phi(a^2)(b)\ldots \phi(a^{n-1})(b)$ to Alice.
 \item Alice computes
 \[
 k_A=\prod_{i=1}^{n-1}\psi(a^i)(y)=\prod_{i=1}^{n-1}\prod_{j=1}^{n-1}\psi(a^i)\phi(a^j)(b).
  \]
  \item Bob computes
  \[
  k_B=\prod_{i=1}^{n-1}\phi(a^i)(x)=\prod_{i=1}^{n-1}\prod_{j=1}^{n-1}\phi(a^i)\psi(a^j)(b).
  \]
\end{itemize}
We require that Alice and Bob pick $\psi$ and $\phi$ so that they commute. If this is done, Alice and Bob have computed a common shared key $k=k_A=k_B$.

The proposal~\cite{HKS} suggests to take $A$ to be a $p$-group (a group of order $p^l$ for some $l\in \mathbb{N}$ and prime $p$) and $B$ to be an elementary abelian $p$-group of order $p^m$.
Thus $B$ may be viewed as an $m$-dimensional vector space over $\mathbb{F}_p$, and so $\mathrm{Aut}(B) = \GL_m(\bF_p)$.
With this choice of platform groups, we can view the protocol as follows.

Define $f(x)=x+x^2+\cdots+x^{n-1}$.
Let $b$ be an $m$-dimensional column vector over $\bF_p$.
Alice and Bob choose private $m\times m$ matrices $J$ and $K$ respectively, using some method so that $f(J)$ and $f(K)$ commute.
In general, and a little more formally, $J=M_A(r_A)$ and $K=M_B(r_B)$ where $M_A$ and $M_B$ are public algorithms which take as input random sequences of coin tosses $r_A$ and $r_B$ respectively (in addition to the public parameters of the scheme).
The algorithms must have the property that the matrices $f(M_A(r_A))$ and $f(M_B(r_b))$ commute for all input sequences $r_A$ and $r_B$ respectively.
The paper~\cite{HKS} suggests some candidates for $M_A$ and $M_B$, but we do not make use of the details of these algorithms in this cryptanalysis.

Alice transmits the column vector $w_A=f(J)b$ to Bob. Bob transmits the column vector $w_B=f(K)b$ to Alice.
The common key $k$ is the column vector defined by
\[
k=f(J)f(K)v=f(J)w_B=f(K)w_A,
\]
the last equality following since $f(J)$ and $f(K)$ commute.

\subsection*{A cryptanalysis}

Suppose an adversary Eve receives $w_A,w_B$ and the public parameters of the scheme.

Let $X$ be any matrix such that $Xb=w_A$, and $X$ commutes with $f(L)$ for all matrices $L$ that can possibly be generated by Bob.
Such a matrix exists since $X=f(J)$ satisfies these conditions.

Note that the conditions on $X$ are linear conditions on the unknown entries of $X$.
This is clear for the condition that $Xb=w_A$.
The commutator condition can be expressed as $Xf(L)=f(L)X$, for matrices $L$ output by the algorithm $M_B$.
To compute the commutator condition on $X$, Eve can run $M_B$ on some random inputs $r_E$ to find suitable matrices $f(L)$ and impose the necessary conditions $Xf(L)=f(L)X$ on $X$.
Since these conditions are linear, the number of random inputs $r_E$ that is required before these necessary conditions become sufficient to imply the commutator condition (at least for an overwhelming proportion of runs of the protocol) is very small.

Since all the conditions on $X$ are linear and easy to find, a suitable matrix $X$ can be computed efficiently.

We claim that $k=Xw_B$. To see this, observe that
\[
Xw_B=Xf(K)b=f(K)Xb=f(K)w_A=f(K)f(J)b=f(J)f(K)b=k.
\]
This means that the adversary can generate the shared key, and the scheme is broken.

\section{The RU Scheme}

We now cryptanalyse a recent key agreement protocol proposed by Romanczuk and Ustimenko~\cite{Romanczuk}.
The protocol works as follows.

Let $\GL_n(\bF_q)$ denote the group of invertible $n\times n$ matrices over a finite field $\bF_q$ of order $q$, and let $\bF_q[x,y]$ denote the polynomial ring over $\bF_q$ in two variables $x$ and $y$.
Let $C,D\in \GL_n(\bF_q)$ be two commuting matrices and let $d\in \bF^n_q$.
The matrices $C,D$ and the vector $d$ are made public.

To agree on a shared key, Alice picks a polynomial $f_A(x,y)\in \bF_q[x,y]$ and sends $w_A= f_A(C,D)d$ to Bob.
Likewise Bob picks a polynomial $f_B(x,y)\in \bF_q[x,y]$ and sends $w_B = f_B(C,D)d$ to Alice.
Alice computes $k_A = f_A(C,D)w_B$, Bob computes $k_B = f_B(C,D)w_A$.
Since $C$ and $D$ commute, the same is true for $f_A(C,D)$ and $f_B(C,D)$, and so their shared key is the vector $k:=k_A=k_B$.

It was not fully specified how the matrices $C,D$ and the polynomials $f_A, f_B$ are generated.
However, the following cryptanalysis applies to any method of generation.

\subsection*{A cryptanalysis}

Suppose a passive adversary Eve receives $w_A,w_B$, and the public quantities $C,D$ and $d$.
Let $X$ be any matrix such that $$XC = CX,\;\;XD=DX,\;\; Xd = w_A.$$
Note that such a matrix exists, since $X = f_A(C,D)$ satisfies these conditions.
Since the conditions on $X$ are all linear, such a matrix is easily found.
Eve can then compute the key as: $$Xw_B = Xf_B(C,D)d = f_B(C,D)Xd = f_B(C,D)w_A = k.$$


\begin{thebibliography}{99}
\bibitem{Baumslag} G. Baumslag and T. Camps and B. Fine and G. Rosenberger and X. Xu, Designing key transport protocols using combinatorial group theory, {Algebraic methods in cryptography: AMS/DMV Joint International Meeting, June 16-19, 2005, Mainz, Germany: International Workshop on Algebraic Methods in Cryptography, November 17-18, 2005, Bochum, Germany}, \emph{Contemporary Mathematics}, \textbf{418}, 35--43, 2006.
\bibitem{Blackburn} S. R. Blackburn, C. Cid and C. Mullan, Group theory in cryptography, to appear in \emph{Proceedings of Groups St Andrews in Bath 2009}, CUP. Available at \url{http://arxiv.org/abs/0906.5545}.
\bibitem{Magma} W. Bosma, J. Cannon and C. Playoust, The Magma algebra system. I. The user language, \emph{J. Symbolic Computation} \textbf{24}, 235–-265, 1997.
\bibitem{HKS} M. Habeeb, D. Kahrobaei and V. Shpilrain, A public key exchange using semidirect products of groups, \emph{Proceedings of SCC 2010}, 137--141, 2010. Available at \url{http://scc2010.rhul.ac.uk/program.php}.
\bibitem{MOV} A.J. Menezes, P.C. van Oorschot and S.A. Vanstone, \emph{Handbook of Applied Cryptography}, CRC Press, Boca Raton, 1997.
\bibitem{Myasnikov} A. Myasnikov, V. Shpilrain and A. Ushakov, \emph{Group-based Cryptography}, Advanced Courses in Mathematics CRM Barcelona (Birkh\"auser, Basel, 2008).
\bibitem{Romanczuk} U. Romanczuk and V. Ustimenko, On the $\mathrm{PSL}_2(q)$, Ramanujan graphs and key exchange protocols. Available at \url{http://aca2010.info/index.php/aca2010/aca2010/paper/viewFile/80/3}.
\end{thebibliography}
\end{document}